\documentclass[11pt]{ip-journal} 
\advance\hoffset by -0.5 truecm
\usepackage{amsmath,amscd}
\usepackage{amssymb}
\usepackage{amsthm}

\usepackage{mathrsfs}
\usepackage{hyperref}

\usepackage[colorinlistoftodos]{todonotes}

\newtheorem{Theorem}{Theorem}[section]

\newtheorem{Lemma}[Theorem]{Lemma}
\newtheorem{Corollary}[Theorem]{Corollary}

\newtheorem{Proposition}[Theorem]{Proposition}

\newtheorem{Definition}[Theorem]{Definition}

\newtheorem{Remark}[Theorem]{Remark}

\numberwithin{equation}{section}

\def \dim{{\mbox {dim}}\,}

\def\V{\mbox{Var}}

\def\R\re
\def\V{\bf V}

\def \re{{\mathbb R}}
\def \mR{{\mathbb R}}

\def \C{{\mathbb C}}
\def \Cm{{\mathbb C}}

\def \V{{\bf V}}

\def \O{{\mathcal O}}

\def \A{{\mathbb A}}
\def \B{{\mathbb B}}

\newcommand{\id}{\mathrm{Id}}

\def \G{\widetilde{G}}

\def \g{{\mathfrak g}}

\newcounter{sidenote}
\begin{document}

\title{The non-Abelian X-ray transform on surfaces}

\author[G.P. Paternain]{Gabriel P. Paternain}
\address{ Department of Pure Mathematics and Mathematical Statistics,
University of Cambridge,
Cambridge CB3 0WB, UK}
\email {g.p.paternain@dpmms.cam.ac.uk}

\author[M. Salo]{Mikko Salo}
\address{Department of Mathematics and Statistics, University of Jyv\"askyl\"a}
\email{mikko.j.salo@jyu.fi}



\begin{abstract} This paper settles the question of injectivity of the non-Abelian X-ray transform on simple surfaces for the general linear group of invertible complex matrices. The main idea is to use a factorization theorem for Loop Groups to reduce to the setting of the unitary group, where energy methods and scalar holomorphic integrating factors can be used. We also show that our main theorem extends to cover the case of an arbitrary Lie group.
\end{abstract}

\maketitle

\section{Introduction}

Given a matrix-valued function $\A$ on a bounded domain $M$ with boundary and a curve $\gamma:[a,b]\to M$ connecting boundary points we may solve the linear matrix differential equation
\[ \dot U + \A(\gamma(t))U = 0, \qquad U(b) = \id. \]


The matrix $C_{\A}(\gamma):=U(a)$ at the boundary, often called the \emph{non-Abelian X-ray transform} or \emph{scattering data} of $\A$, is expected to give good information about the function $\A$ inside $M$ once we have enough curves $\gamma$ travelling through $M$.
We show that indeed one can recover $\A$ from $C_{\A}$ when the curves $\gamma$ are the geodesics of a surface $M$ with strictly convex boundary, no trapped geodesics and no caustics.

\smallskip

\subsection{Statement of results} Let $(M,g)$ be a compact oriented Riemannian surface (i.e.\ two-dimensional manifold) with smooth boundary and let $SM = \{ (x,v)\in TM:\ g(v,v) = 1 \}$ be the unit tangent bundle. The geodesics starting at $\partial M$ and moving into $M$ can be parametrized by the {\it influx boundary}
 \[\partial_{+}SM:=\{(x,v)\in SM:\;x\in\partial M,\;g(v,\nu)\leq 0\}\]
where $\nu$ is the outer unit normal to $\partial M$. Given $(x,v)\in SM$, we let $\gamma_{x,v}(t)$ be the geodesic starting at $x$ with velocity $v$.
We will assume that the surface is {\it non-trapping}, which means that the time $\tau(x,v)$ when the geodesic $\gamma_{x,v}$ exits $M$ is finite for all $(x,v)\in SM$. Moreover, we will assume that $\partial M$ is strictly convex, meaning that the second fundamental form of $\partial M\subset M$ is positive definite.
This is already enough to imply that $M$ is diffeomorphic to a closed disc (cf. \cite{PSU20}). If in addition $(M,g)$ has no conjugate points we say that the surface is {\it simple}.

Our object of interest is the {\it non-Abelian X-ray transform} associated with a pair given by a connection $A$ and a matrix valued field $\Phi$. Let $G$ be a Lie group of matrices with Lie algebra $\g$. The connection $A$ is just an element of $\Omega^{1}(M,\g)$, namely a smooth $\g$-valued 1-form and $\Phi\in C^{\infty}(M,\g)$. Given such a pair $(A,\Phi)$ (the reader might wish to think of $A$ as a Yang-Mills potential, and $\Phi$ as a Higgs field), 
and $\gamma_{x,v}$ a geodesic determined by $(x,v)$ in the influx boundary, we consider the matrix ordinary differential equation along $\gamma_{x,v}$:
\begin{align*}
    \dot U + \left[ A_{\gamma_{x,v}(t)}(\dot{\gamma}_{x,v}(t))+\Phi(\gamma_{x,v}(t)) \right] U = 0, \qquad U(\tau(x,v)) = \id. 
\end{align*}
Since $A$ and $\Phi$ take values in $\g$, the solution $U$ maps $U:[0,\tau(x,v)]\to G$ (see e.g.\ \cite[Proposition 5.3 in Chapter 1]{Taylor}). The scattering
data along $\gamma_{x,v}$ is defined as $C_{A,\Phi}(\gamma_{x,v}):=U(0)$. Observe that when $A$ and $\Phi$ are scalar (i.e. $\C$-valued), we obtain 
\[
\log \,U(0) = \int_0^{\tau(x,v)} [A_{\gamma_{x,v}(t)}(\dot{\gamma}_{x,v}(t))+\Phi(\gamma_{x,v}(t))]\ dt,
\]
which is the classical X-ray/Radon transform of $A+\Phi$ along the curve $\gamma_{x,v}$.
Considering the collection of all such data makes up the {\em scattering data} (or {\em non-Abelian X-ray transform}) of the pair $(A,\Phi)$, viewed here as a map 
\begin{align*}
    C_{A,\Phi}\colon \partial_+ SM\to G.
\end{align*}

We are concerned with the recovery of $(A,\Phi)$ from $C_{A, \Phi}$. The problem exhibits a natural gauge equivalence associated with the gauge group $\mathcal G$ given by those smooth $u:M\to G$
such that $u|_{\partial M}=\id$. The gauge group $\mathcal G$ acts on pairs (from the right) as follows:
\[(A,\Phi)\cdot u=(u^{-1}du+u^{-1}Au,u^{-1}\Phi u).\]
It is straightforward to check that for any $u\in \mathcal G$,
\[C_{(A,\Phi)\cdot u}=C_{A,\Phi}.\]
The geometric inverse problem consists in showing that the non-Abelian X-ray transform
\[(A,\Phi)\mapsto C_{A,\Phi}\]
 is injective up the action of $\mathcal G$. The present paper settles this question for the general linear group
$GL(n,\C)$ when $M$ is a simple surface. We shall indistinctly denote the set of (complex) $n\times n$ matrices by $\C^{n\times n}$ or
$\mathfrak{gl}(n,\C)$ if we wish to think of matrices as the Lie algebra of the general linear group $GL(n,\C)$.

\begin{Theorem} \label{thm:nonabelian}
Let $(M,g)$ be a simple surface. Suppose we are given pairs $(A,\Phi)$ and $(B,\Psi)$ with
$A,B\in \Omega^{1}(M,\mathfrak{gl}(n,\C))$ and
$\Phi,\Psi\in C^{\infty}(M,\mathfrak{gl}(n,\C))$. If
\[C_{A,\Phi}=C_{B,\Psi},\]
then there is $u\in \mathcal G$ such that $(A,\Phi)\cdot u=(B,\Psi)$.
\end{Theorem}

Note that the theorem implies in particular that scattering rigidity just for matrix fields does not have a gauge. Indeed, if $C_{\Phi}=C_{\Psi}$, where $\Phi$ and $\Psi$ are two matrix fields, Theorem \ref{thm:nonabelian} applied with $A=B=0$ implies that $u=\id$ and thus $\Phi=\Psi$.

The non-linear inverse problem resolved in Theorem \ref{thm:nonabelian} is closely related to a linear inverse problem involving an attenuated X-ray transform. The relationship is via a pseudo-linearization identity and is well known;
we explain this relationship in detail in Section \ref{section:pseudo}, but for now we simply refer to equation \eqref{eq:pseudo-linearization} below.
We now state the solution to the relevant linear inverse problem. Given a pair $(A,\Phi)$ taking values in $\mathfrak{gl}(n,\C)$ and $f\in C^{\infty}(SM,\C^n)$, consider the unique solution $u(t)$ to the vector valued
ordinary differential equation
\begin{align*}
    \dot u + \left[ A_{\gamma_{x,v}(t)}(\dot{\gamma}_{x,v}(t))+\Phi(\gamma_{x,v}(t)) \right] u = -f(\gamma_{x,v}(t),\dot{\gamma}_{x,v}(t)), \quad u(\tau(x,v)) = 0. 
\end{align*}
We define the {\it attenuated X-ray transform} of $f$ as
\[I_{A,\Phi}(f)(x,v):=u(0)\]
where $(x,v)\in \partial_{+}SM$.  We have:

\begin{Theorem} \label{thm:main}
Let $M$ be a simple surface and consider an arbitrary attenuation pair $(A,\Phi)$ with $A\in \Omega^{1}(M,\mathfrak{gl}(n,\C))$ and
$\Phi\in C^{\infty}(M,\mathfrak{gl}(n,\C))$. 
Assume that $f:SM\to\C^n$ is a smooth function of the form
$F(x)+\alpha_{x}(v)$, where $F:M\to\C^n$ is a smooth function
and $\alpha$ is a $\C^n$-valued 1-form. If $I_{A,\Phi}(f)=0$, then
$F = \Phi p$ and $\alpha=dp+Ap$, where $p:M\to\C^n$ is a smooth
function with $p|_{\partial M}=0$.
\end{Theorem}

Theorems \ref{thm:nonabelian} and \ref{thm:main} were proved in \cite{PSU12} when the pair $(A,\Phi)$ takes values in the Lie algebra of the unitary group $\mathfrak{u}(n)$ (skew-hermitian matrices). The main idea in
the present paper is to use a basic factorization theorem for Loop Groups to perform a transformation that takes
the problem for the Lie algebra $\mathfrak{gl}(n,\C)$ to the problem for the Lie algebra $\mathfrak{u}(n)$
that we already know how to solve. The method of proof in \cite{PSU12} was based in ``moving across" the scheme of proof of the well-known Kodaira vanishing theorem in Complex Geometry to the transport problem relevant for the non-Abelian X-ray transform. In particular an important energy identity was used (the Pestov identity, analogous in some sense to the Weitzenb\"ock identity, but involving instead the geodesic vector field)
and this identity develops unmanageable terms once the pair $(A,\Phi)$ stops taking values in $\mathfrak{u}(n)$; in other words
we need to deal with a dissipative situation as far as energy identities is concerned. A fix to this problem was implemented by the authors
in \cite{PS18}, but it comes at a cost: we need to assume negative curvature. The upgrade from negative curvature
to no conjugate points that the present paper provides seems out of reach using the estimates in \cite{PS18}.
The structure theorem for Loop Groups that we use is the infinite dimensional version of the familiar fact
that asserts that an invertible matrix is the product of an upper triangular matrix and a unitary matrix. It is perhaps the most basic of the factorization theorems that include also the Birkhoff and Bruhat factorizations \cite[Chapter 8]{PS86}.

It turns out that Theorem \ref{thm:nonabelian} is enough to resolve the problem of injectivity of the non-Abelian X-ray transform for an arbitrary Lie group $G$; we explain this in Section \ref{section:lieg}, see Theorem \ref{thm:nonabelian2} below.

Finally, as a corollary we deduce that it is possible to detect purely from boundary measurements whether a matrix-valued field takes values in
the set of skew-hermitian matrices:

\begin{Corollary} Let $(M,g)$ be a simple surface and $\Phi \in  C^{\infty}(M,\mathfrak{gl}(n,\C))$.
Then $C_{\Phi}$ takes values in the unitary group iff $\Phi^{*}=-\Phi$, where $\Phi^*$ denotes the conjugate transpose of $\Phi$.
\end{Corollary}

\begin{proof} From the definition of the scattering data we see that $C^*_{\Phi}=C_{-\Phi^{*}}^{-1}$.
If $C_{\Phi}$ is unitary we have $C_{\Phi}=C_{-\Phi^*}$ and Theorem \ref{thm:nonabelian} gives $\Phi=-\Phi^*$.
\end{proof}

\subsection{Motivation} The non-Abelian X-transform $(A,\Phi)\mapsto C_{A,\Phi}$ appears naturally in several contexts. For instance, when $\Phi=0$, $C_{A}$ represents the parallel transport of the connection $A$ along geodesics connecting boundary points and the injectivity question for the non-Abelian X-ray transform reduces to the question of recovering a connection up to gauge from its parallel transport along a distinguished set of curves, i.e.\ the geodesics of the metric $g$. If $A\in \Omega^{1}(M,\mathfrak{u}(n))$, we may consider the twisted or connection Laplacian $d_{A}^*d_{A}$, where $d_{A}=d+A$. Egorov's theorem for the connection Laplacian
naturally produces the parallel transport of $A$ along geodesics of $g$ as a high energy limit, cf.\ \cite[Proposition 3.3]{JS07}, and this data can also be obtained from the corresponding wave equation following \cite{OSSU, Uhlmann_scattering}. 
When $A=0$ and $\Phi\in C^{\infty}(M,\mathfrak{so}(3))$, the non-Abelian X-ray transform $\Phi\mapsto C_{\Phi}$ arises in Polarimetric Neutron Tomography \cite{DLSS,Hetal}, a new tomographic method designed to detect magnetic fields inside materials by probing them with neutron beams. The case of pairs $(A,\Phi)$ arises in the literature on solitons, mostly in the context of the Bogomolny equations in $2+1$ dimensions \cite{MZ81,W88}.
Applications to coherent quantum tomography are given in \cite{Il}.
We refer to \cite{Novikov_nonabelian} for a recent survey on the non-Abelian X-ray transform and its applications.

\subsection{Comparison with existing literature}

We first mention that there is a substantial difference between the case $\dim M = 2$ considered in this article and the case $\dim M \geq 3$. In fact, in three and higher dimensions the inverse problems considered in Theorems \ref{thm:nonabelian} and \ref{thm:main} are formally overdetermined, whereas in two dimensions they are formally determined (one attempts to recover functions depending on $d$ variables from data depending on $2d-2$ variables). When $\dim M \geq 3$, results corresponding to Theorems \ref{thm:nonabelian} and \ref{thm:main} are proved in \cite{No} in the case of $\mR^3$ and in \cite{PSUZ} on compact strictly convex manifolds admitting a strictly convex function, based on the method introduced in \cite{UV}.

We will now focus on earlier results for $\dim M = 2$. As we have already mentioned, Theorems \ref{thm:nonabelian} and \ref{thm:main} were proved in \cite{PSU12} when the pair $(A,\Phi)$ takes values in $\mathfrak{u}(n)$. There are several other important contributions that we now briefly review. To organise the discussion we consider
two scenarios: the Euclidean case and non-Euclidean one.
When $(M,g)$ is a subset of $\re^2$ with the Euclidean metric, the literature is extensive, particularly
in the abelian case $n=1$, where a result like Theorem \ref{thm:main} is simply the statement of injectivity of the attenuated Radon transform relevant in the imaging modality SPECT. In this case we limit ourselves to a discussion involving the genuinely non-Abelian situation ($n\geq 2$). The results here tend to be formulated in all
$\re^2$ and in parallel-beam geometry taking advantage that geodesics are just straight lines.
In \cite{No}, R.{ Novikov considers pairs $(A,\Phi)$ that are not compactly supported but have suitable decay conditions at infinity and establishes local uniqueness of the trivial pair and gives examples
in which global uniqueness fails (existence of ''ghosts''). G.\ Eskin in \cite{E} considers compactly supported pairs and shows injectivity as in Theorem \ref{thm:nonabelian}.  His proof relies on a delicate result proved in \cite{ER04} on the existence of {\it matrix} holomorphic (in the vertical direction) integrating factors. We note that our proof of Theorem \ref{thm:nonabelian} replaces this delicate step by the use of the Loop Group factorization theorem and the proof via energy identities in \cite{PSU12}
that only requires the existence of {\it scalar} holomorphic integrating factors. These are supplied
via microlocal analysis of the normal operator of the standard X-ray transform by \cite{PU05}.
In the Euclidean setting, we also mention the result of Finch and Uhlmann in \cite{FU} that establishes
injectivity up to gauge for unitary connections assuming that they have small curvature.

In the non-Euclidean setting, as far as we are aware the first contributions appear in \cite{V1,V2,Sh1}, but these results have restrictions on the size of the pairs $(A,\Phi)$.  Theorem \ref{thm:main} for $A=0$ and $n=1$ was proved in \cite{SU11}. Genericity results and Fredholm alternatives for the problem are given in \cite{MoPa,Z}. As we have already mentioned, \cite{PS18} proves Theorem \ref{thm:nonabelian} assuming negative Gaussian curvature. The existence of {\it matrix} holomorphic integrating factors for any simple surface has been recently settled in \cite{BP21}; however this result requires the key input from the present paper (cf. Lemma \ref{lemma:miracle} below). In \cite{BP21} the authors also give a full characterization of the range of the non-Abelian X-ray transform.
The problem can also be considered for closed surfaces, cf.\ \cite{P1} for a survey that includes these cases.


\subsection*{Acknowledgements} GPP was supported by EPSRC grant EP/R001898/1 and the Leverhulme trust.
MS\ was supported by the Academy of Finland (Finnish Centre of Excellence in Inverse Modelling and Imaging, grant numbers 312121 and 309963) and by the European Research Council under Horizon 2020 (ERC CoG 770924).

\section{Preliminaries}

This section provides some well-known background material and it may all be found in \cite{GK,SiTh}; for a recent presentation and its relevance to geometric inverse problems in two dimensions we refer to \cite{PSU20}.

Let $(M,g)$ be a compact oriented two dimensional Riemannian manifold with smooth boundary
$\partial M$. Let $X$ denote the geodesic vector field, i.e.\ the infinitesimal generator of the geodesic flow $\varphi_t$ acting on the unit
circle bundle $SM$. The latter is a compact 3-manifold with boundary given by $\partial SM=\{(x,v)\in SM:\;x\in \partial M\}$. 
Since $M$ is assumed oriented there is a circle action on the fibers of $SM$ with infinitesimal generator $V$ called the {\it vertical vector field}. It is possible to complete the pair $X,V$ to a global frame
of $T(SM)$ by considering the vector field $X_{\perp}:=[X,V]$. There are two additional structure equations given by $X=[V,X_{\perp}]$ and $[X,X_{\perp}]=-KV$
where $K$ is the Gaussian curvature of the surface. Using this frame we can define a Riemannian metric on $SM$ by declaring $\{X,X_{\perp},V\}$ to be an orthonormal basis and the volume form of this metric will be denoted by $d\Sigma^3$. The fact that $\{ X, X_{\perp}, V \}$ are orthonormal together with the commutator formulas implies that the Lie derivative of $d\Sigma^3$ along the three vector fields vanishes.

If $x = (x_1, x_2)$ are isothermal coordinates in $(M,g)$ so that the metric has the form $g = e^{2\lambda(x)} \,dx^2$ and if  $\theta$ is the angle between $v$ and $\partial_{x_1}$, then in the $(x,\theta)$ coordinates in $SM$ the vector fields have the explicit formulas
\begin{align*}
X &= e^{-\lambda}\left(\cos\theta\frac{\partial}{\partial x_{1}}+
\sin\theta\frac{\partial}{\partial x_{2}}+
\left(-\frac{\partial \lambda}{\partial x_{1}}\sin\theta+\frac{\partial\lambda}{\partial x_{2}}\cos\theta\right)\frac{\partial}{\partial \theta}\right), \\
X_{\perp} &= -e^{-\lambda}\left(-\sin\theta\frac{\partial}{\partial x_{1}}+
\cos\theta\frac{\partial}{\partial x_{2}}-
\left(\frac{\partial \lambda}{\partial x_{1}}\cos\theta+\frac{\partial \lambda}{\partial x_{2}}\sin\theta\right)\frac{\partial}{\partial \theta}\right), \\
V &= \frac{\partial}{\partial\theta}.
\end{align*}

Given functions $u,v:SM\to \C^n$ we consider the
inner product
\[(u,v) =\int_{SM}\langle u,v\rangle_{\mathbb \C^n}\,d\Sigma^3.\]

The space $L^{2}(SM,\C^n)$ decomposes orthogonally
as a direct sum
\[L^{2}(SM,\C^n)=\bigoplus_{k\in\mathbb Z}H_{k}\]
where $H_k$ is the eigenspace of $-iV$ corresponding to the eigenvalue $k$.
A function $u\in L^{2}(SM,\C^n)$ has a Fourier series expansion
\[u=\sum_{k=-\infty}^{\infty}u_{k},\]
where $u_{k}\in H_k$. Let $\Omega_{k}=C^{\infty}(SM,\C^n)\cap H_{k}$.

\begin{Definition} A function $u:SM\to\C^n$ is said to be (fibre-wise) holomorphic 
 if $u_{k}=0$ for all $k<0$. Similarly, $u$ is said to be (fibre-wise) antiholomorphic if $u_{k}=0$ for all $k>0$.
\end{Definition}

As in \cite{GK} we introduce the following
first order operators 
$$\eta_{+},\eta_{-}:C^{\infty}(SM,\C^n)\to
C^{\infty}(SM,\C^n)$$ given by
\[\eta_{+}:=(X+iX_{\perp})/2,\;\;\;\;\;\;\eta_{-}:=(X-iX_{\perp})/2.\]
Clearly $X=\eta_{+}+\eta_{-}$. 
We have
\[\eta_{+}:\Omega_{m}\to \Omega_{m+1},\;\;\;\;\eta_{-}:\Omega_{m}\to \Omega_{m-1},\;\;\;\;(\eta_{+})^{*}=-\eta_{-}.\]
In particular, $X$ has the following important mapping property
\[X:\oplus_{k\geq 0}\Omega_{k}\to \oplus_{k\geq -1}\Omega_{k}.\]
(For any $I\subset\mathbb{Z}$, $\oplus_{k\in I}\Omega_{k}$ denotes the set of smooth functions $u$ such that $u_{k}=0$ for $k\notin I$.)
We will often use all of the above for smooth functions taking values in complex matrices $\mathfrak{gl}(n,\C)$
and we will not make any distinction in the notation as it will become clear from the context.

\section{The pseudo-linearization identity}
\label{section:pseudo}

Let $(M,g)$ be a compact non-trapping manifold with strictly convex boundary and let  $\A\in C^{\infty}(SM,\C^{n\times n})$. Consider $(M,g)$ isometrically embedded in a closed manifold $(N,g)$ and extend $\A$ smoothly to $N$. Under these assumptions, $\A$ on $N$ defines a {\it smooth} cocycle $C$ over
the geodesic flow $\varphi_t$ of $(N,g)$. The cocycle takes values in the group $GL(n,\C)$ and
is defined as follows: let
$C:SN\times \mathbb{R}\to GL(n,\C)$ be determined by the following matrix ODE along the orbits of the geodesic flow
\[\frac{d}{dt}C(x,v,t)+\A(\varphi_{t}(x,v))C(x,v,t)=0,\;\;\;\;\;C(x,v,0)=\mbox{\rm Id}.\]
The function $C$ is a {\it cocycle}: 
\[C(x,v,t+s)=C(\varphi_{t}(x,v),s)\,C(x,v,t)\]
for all $(x,v)\in SN$ and $s,t\in\mathbb{R}$. 

Consider a slightly larger compact manifold $M_0$ engulfing $M$ so that $(M_0,g)$ is still non-trapping with strictly convex boundary and let $\tau_0$ be the exit time of $M_0$.  The next lemma shows that the equation $XR + \A R = 0$ in $SM$ has a smooth solution.

\begin{Lemma} The function $R:SM\to GL(n,\C)$ defined by
\[R(x,v):=[C(x,v,\tau_{0}(x,v))]^{-1},\]
is smooth and satisfies
\[XR+\A R=0.\]
\label{lemma:matrixintegrating}
\end{Lemma}

\begin{proof} Since $\tau_0|_{SM}$ is smooth and the cocycle $C$ is smooth, the smoothness of $R$ follows right away. To check that $R$ satisfies the stated equation, we use that $\tau_{0}(\varphi_{t}(x,v))=\tau_{0}(x,v)-t$ together with the cocycle property to obtain
\[R(\varphi_{t}(x,v))=[C(\varphi_{t}(x,v),\tau_0(\varphi_{t}(x,v))]^{-1}=C(x,v,t)[C(x,v,\tau_{0}(x,v))]^{-1}.\]
Diiferentiating at $t=0$ yields
\[XR=-\A R\]
as desired.
\end{proof}

Let us define the {\it outflux boundary} by
 \[\partial_{-}SM:=\{(x,v)\in SM:\;x\in\partial M,\;g(v,\nu)\geq 0\}.\]
From the proof of Lemma \ref{lemma:matrixintegrating} we see that the function
\[U_{+}(x,v):=[C(x,v,\tau(x,v)]^{-1}\]
solves 
\begin{equation}
\left\{\begin{array}{ll}
XU_{+}+\A U_{+}=0,\\
U_{+}|_{\partial_{-}SM}=\id.\\
\end{array}\right.
\label{eq:weightplus}
\end{equation}

\begin{Definition} The {\it scattering data} of $\A$ is the map $C_{\A,+}:\partial_{+}SM\to GL(n,\C)$ given by
\[C_{\A,+}:=U_{+}|_{\partial_{+}SM}.\]
We shall also call $C_{\A,+}$ the non-abelian X-ray transform of $\A$.
\end{Definition}
Note that $C_{\A,+}\in C^{\infty}(\partial_{+}SM,\C^{n\times n})$. We can also consider the unique solution of

\begin{equation}
\left\{\begin{array}{ll}
XU_{-}+\A U_{-}=0,\\
U_{-}|_{\partial_{+}SM}=\id\\
\end{array}\right.
\label{eq:weightminus}
\end{equation}
and define scattering data $C_{\A,-}:\partial_{-}SM\to GL(n,\C)$ by setting
\[C_{\A,-}:=U_{-}|_{\partial_{-}SM}.\]
As discussed in \cite[Section 3]{PSU12}, both quantities are related by
\begin{equation}
C_{\A,-}=[C_{\A,+}]^{-1}\circ\alpha,
\label{eq:relplusminus}
\end{equation}
where $\alpha:\partial SM\to \partial SM$ is the scattering relation of the metric $g$.
In this paper we only work with $C_{\A,+}$ and from now on we drop the subscript $+$ from the notation.

\subsection{Attenuated X-ray transforms}Recall that in the scalar case, the attenuated ray transform $I_a f$ of a function $f \in C^{\infty}(SM,\C)$ with attenuation coefficient $a \in C^{\infty}(SM,\C)$ can be defined as the integral 
\[
I_a f(x,v) := \int_0^{\tau(x,v)} f(\varphi_t(x,v)) \text{exp}\left[ \int_0^t a(\varphi_s(x,v)) \,ds \right] dt, \quad (x,v) \in \partial_+ SM.
\]
Alternatively, we may set $I_a f := u|_{\partial_+ SM}$ where $u$ is the unique solution of the transport equation 
\[
Xu + au = -f \ \ \text{in $SM$}, \quad u|_{\partial_- SM} = 0.
\]

The last definition generalizes without difficulty to the case of a general attenuation $\A$. Let $f \in C^{\infty}(SM,\C^n)$ be a vector valued function and consider the following transport equation for $u: SM \to \C^n$, 
\[
Xu + {\A}u = -f \ \ \text{in $SM$}, \quad u|_{\partial_- SM} = 0.
\]
On a fixed geodesic the transport equation becomes a linear ODE with zero final condition, and therefore this equation has a unique solution denoted by $u^f$.

\begin{Definition} \label{def:aXray}
The attenuated X-ray transform of $f \in C^{\infty}(SM,\C^n)$ is given by 
\[
I_{\A} f := u^f|_{\partial_+ SM}.
\]
\end{Definition}

It is a simple task to write an integral formula for $u^f$ using a matrix integrating factor as in Lemma \ref{lemma:matrixintegrating}.

\begin{Lemma} If $R:SM\to GL(n,\C)$ solves $XR+\A R=0$, then 
\[u^f(x,v)=R(x,v)\int_{0}^{\tau(x,v)}(R^{-1}f)(\varphi_{t}(x,v))\,dt\;\;\text{for}\;(x,v)\in SM.\]
\label{lemma:formulau}
\end{Lemma}

\begin{proof} A computation using $XR^{-1}=R^{-1}\A$ (which follows easily from $XR+\A R=0$)
and $Xu^f+\A u^f=-f$ yields
\[X(R^{-1}u^{f})=(XR^{-1})u^{f}+R^{-1}u^{f}=-R^{-1}f.\]
Since $R^{-1}u^{f}|_{\partial_{-}SM}=0$, the lemma follows.
\end{proof}

\subsection{Pseudo-linearization identity} Given two functions $\A,\B\in C^{\infty}(SM,\C^{n\times n})$ we would like to have a formula relating $C_{\A}$ and $C_{\B}$ with certain attenuated X-ray transform. The following argument is quite similar to the one in \cite[Section 8]{PSU12}. We first introduce the map
$E(\A,\B):SM\to \text{End}(\C^{n\times n})$ given by
\[E(\A,\B)U:=\A U-U\B.\]
Here, $\text{End}(\C^{n\times n})$ denotes the linear endomorphisms of $\C^{n\times n}$.

\begin{Proposition}
Let $(M,g)$ be a compact non-trapping manifold with strictly convex boundary.
Given $\A,\B\in C^{\infty}(SM,\C^{n\times n})$, we have
\begin{equation}
C_{\A}C_{\B}^{-1}=\id+I_{E(\A,\B)}(\A-\B),
\label{eq:pseudo-linearization}
\end{equation}
where $I_{E(\A,\B)}$ denotes the attenuated X-ray transform with attenuation $E(\A,\B)$ as given in Definition \ref{def:aXray}.
\end{Proposition}

\begin{proof} Consider the fundamental solutions for both $\A$ and $\B$, namely
\[\left\{\begin{array}{ll}
XU_{\A}+\A U_{\A}=0,\\
U_{\A}|_{\partial_{-}SM}=\id,\\
\end{array}\right.\]
and
\[\left\{\begin{array}{ll}
XU_{\B}+\B U_{\B}=0,\\
U_{\B}|_{\partial_{-}SM}=\id.\\
\end{array}\right.\]
Let $W:=U_{\A}U_{\B}^{-1}-\id$. A direct computation shows that
\[\left\{\begin{array}{ll}
XW+\A W-W\B=-(\A-\B),\\
W|_{\partial_{-}SM}=0.\\
\end{array}\right.\]
By definition of $I_{E(\A,\B)}$ we have
\[ I_{E(\A,\B)}(\A-\B)=W|_{\partial_{+}SM}\]
and since by construction $W|_{\partial_{+}SM}=C_{\A}C_{\B}^{-1}-\id$, the proposition follows.
\end{proof}

\begin{Remark} \label{remark:gaugeU}
Note that the function $U:=U_{\A}U_{\B}^{-1}$ satisfies
\[\left\{\begin{array}{ll}
\B=U^{-1}XU+U^{-1}\A U,\\
U|_{\partial_{-}SM}=\id.\\
\end{array}\right.\]
\end{Remark}

Using the identity given in Remark \ref{remark:gaugeU} we can establish when two attenuations $\A,\B\in C^{\infty}(SM,\C^{n\times n})$ have the same
non-Abelian X-ray data:

\begin{Proposition} \label{proposition:gaugeU}
Let $(M,g)$ be a compact non-trapping manifold with strictly convex boundary.
Given $\A,\B\in C^{\infty}(SM,\C^{n\times n})$, we have $C_{\A}=C_{\B}$ if and only if
there exists a smooth $U:SM\to GL(n,\C)$ with $U|_{\partial SM}=\id$ and such that
\[\B=U^{-1}XU+U^{-1}\A U.\]
\end{Proposition}

\begin{proof} If such a smooth function $U$ exists, then the function $V=UU_{\B}$ satisfies $XV+\A V=0$ and $V|_{\partial_{-}SM}=\id$ and thus $V=U_{\A}$ and consequently $C_{\A}=C_{\B}$. Conversely, if the non-abelian X-ray transforms agree, the function
$W$ in the proof of Proposition \ref{eq:pseudo-linearization} has zero boundary value and by \cite[Proposition 5.2]{PSU12} it must be smooth. Hence $U=W+\id$ is smooth and by Remark \ref{remark:gaugeU} it satisfies the required equation.
\end{proof}

\section{A factorization theorem from Loop groups}

The main new input in the proof of Theorem \ref{thm:main} is a well known factorization theorem for Loop Groups.
Let us state it precisely following the presentation in \cite[Chapter 8]{PS86}.

Let us denote by $LGL_{n}(\C)$ the set of all smooth maps $\gamma:S^{1}\to GL(n,\C)$. The set has a natural structure of an infinite dimensional Lie group as explained in \cite[Section 3.2]{PS86}. This group contains several subgroups which are relevant for us. We shall denote by $L^{+}GL_{n}(\C)$ the subgroup consisting of those
loops $\gamma$ which are boundary values of holomorphic maps
\[\gamma:\{z\in\C:\,\,|z|<1\}\to GL(n,\C).\]
We let $\Omega U_{n}$ denote the set of smooth loops $\gamma:S^{1}\to U(n)$ such that $\gamma(1)=\id$, where $U(n)$ denotes the unitary group.

The result we shall use is \cite[Theorem 8.1.1]{PS86}, the first of three well-known factorization theorems (the second is Birkhoff's factorization equivalent to the classification of holomorphic vector bundles over $S^{2}$).
A PDE-based proof of this result may also be found in \cite{D92}.

\begin{Theorem} Any loop $\gamma\in LGL_{n}(\C)$ can be factorized uniquely 
\[\gamma=\gamma_{u}\cdot \gamma_{+},\]
with $\gamma_{u}\in\Omega U_{n}$ and $\gamma_{+}\in L^{+}GL_{n}(\C)$. In fact, the product map
\[\Omega U_{n}\times L^{+}GL_{n}(\C)\to LGL_{n}(\C)\]
is a diffeomorphism.
\label{thm:PS}
\end{Theorem}

Before discussing the application of this result to our geometric setting a couple of remarks are in order.
Given a complex $n\times n$ matrix $A$ we shall denote by $A^{T}$, $\overline{A}$ and $A^{*}$, its transpose, its conjugate and its conjugate-transpose respectively. Given $\gamma\in LGL_{n}(\C)$, using the theorem above we may write uniquely $\gamma^T=\gamma_{u}\cdot \gamma_{+}$ and after taking transpose we have
$\gamma=\gamma_{+}^T\cdot \gamma_{u}^T$. Since $\gamma_{+}^T\in  L^{+}GL_{n}(\C)$ and
$\gamma_{u}^T\in \Omega U_{n}$, Theorem \ref{thm:PS} also gives that the product map
\[L^{+}GL_{n}(\C)\times \Omega U_{n}\to LGL_{n}(\C)\]
is a diffeomorphism. We may also consider the subgroup $L^{-}GL_{n}(\C)$ consisting of those
loops $\gamma$ which are boundary values of {\it anti-holomorphic} maps
\[\gamma:\{z\in\C:\,\,|z|<1\}\to GL(n,\C).\]
After conjugating, Theorem \ref{thm:PS} also gives that the product maps
\[\Omega U_{n}\times L^{-}GL_{n}(\C)\to LGL_{n}(\C),\;\;\;\;L^{-}GL_{n}(\C)\times \Omega U_{n}\to LGL_{n}(\C)\]
are diffeomorphisms.

Consider now a compact non-trapping surface $(M,g)$ with strictly convex boundary. It is well known that
such surfaces are diffeomorphic to a disc, cf. \cite{PSU20}. Thus after picking global isothermal coordinates
we may assume that $M$ is the unit disc in the plane and the metric has the form $e^{2\lambda}(dx_{1}^2+dx_{2}^{2})$ where $\lambda$ is a smooth real-valued function of $x=(x_{1},x_{2})$. This gives coodinates
$(x_{1},x_{2},\theta)$ on $SM=M\times S^{1}$, where $\theta$ is the angle between a unit vector and $\partial_{x_{1}}$.

We wish to use the factorization theorem for Loop Groups in the following form:

\begin{Theorem} Given a smooth map $R:SM\to GL(n,\C)$, there are smooth maps $U:SM\to U(n)$
and $F:SM\to GL(n,\C)$ such that $R=FU$ and $F$ is fibre-wise holomorphic with fibre-wise holomorphic inverse.
We may also factorize $R$ as $R=\tilde{F}\tilde{U}$ where $\tilde{U}:SM\to U(n)$ is smooth
and $\tilde{F}:SM\to GL(n,\C)$ is smooth, fibre-wise anti-holomorphic with fibre-wise anti-holomorphic inverse.
\label{thm:PSforuse}
\end{Theorem}

\begin{proof} We only do the proof for $F$ holomorphic (the anti-holomorphic case is entirely analogous).
We regard $R$ as map $R:M\times S^{1}\to GL(n,\C)$ and we claim that we have a 
smooth map 
\[
M \ni x \mapsto R(x,\cdot) \in LGL_{n}(\C).
\]
To prove this, fix $x_0 \in M$ and let $\rho_0 = R(x_0, \,\cdot\,)$. Following \cite[Section 3.2]{PS86}, we may consider a neighborhood $\rho_0 \mathcal{U}$ of $\rho_0$ in $LGL_n(\C)$ where $\mathcal{U} = \exp(C^{\infty}(S^1, \breve{U}))$ and $\breve{U}$ is a small neighborhood of the zero matrix in $\C^{n \times n}$. Now $x \mapsto R(x,\,\cdot\,)$ is smooth near $x_0$ if the map $x \mapsto \log(\rho(x_0)^{-1} R(x,\,\cdot\,))$, where $\log$ is the standard logarithm for matrices close to $\id$, is smooth near $x_0$ as a map from $\re^2$ to the topological vector space $C^{\infty}(S^1, \C^{n \times n})$. The last fact follows easily from the smoothness of $R$.

Using Theorem \ref{thm:PS} in the form that says that the map
\[L^{+}GL_{n}(\C)\times \Omega U_{n}\to LGL_{n}(\C)\]
is a diffeomorphism we may write for each $x\in M$, $R(x,\cdot)=F(x,\cdot)U(x,\cdot)$, where $U$ takes values in the unitary group and $F$ is fibre-wise holomorphic with fibre-wise holomorphic inverse.
Moreover, maps $M\ni x\mapsto F(x,\cdot)$ and $M\ni x\mapsto U(x,\cdot)$ are smooth and the theorem follows. 
\end{proof}

\section{Proof of Theorems \ref{thm:nonabelian} and \ref{thm:main}}

We start with an elementary lemma.

\begin{Lemma} Let $\B\in C^{\infty}(SM,\C^{n\times n})$. If $\B$ is skew-hermitian, i.e.\ $\B\in C^{\infty}(SM,\mathfrak{u}(n))$, and $\B\in \oplus_{k\geq -1}\Omega_{k}$, then $\B\in \Omega_{-1}\oplus\Omega_{0}\oplus \Omega_{1}$ and $\B^{*}_{-1}=-\B_{1}$ and $\B_{0}^{*}=-\B_{0}$.
\label{lemma:easy}
\end{Lemma}

\begin{proof} Expanding $\B$ in Fourier modes we may write $\B=\sum_{k\geq -1}\B_{k}$ and hence
\[\B^*=\left(\sum_{k\geq -1}\B_{k}\right)^{*}=\sum_{k\geq -1}\B_{k}^{*}, \qquad -\B =-\sum_{k\geq -1}\B_{k}.\]
Since $\B^* = -\B$ and $\B_{k}^*\in \Omega_{-k}$, the lemma follows.
\end{proof}

The next lemma is what makes the proof of Theorem \ref{thm:main} possible.

\begin{Lemma} Let $(M,g)$ be a compact non-trapping surface with strictly convex boundary. 
Let $\A\in C^{\infty}(SM,\mathfrak{gl}(n,\C))$ and assume $\A\in \oplus_{k\geq -1}\Omega_{k}$.
Let $R:SM\to GL(n,\C)$ be a smooth function solving
$XR+\A R=0$ (as given by Lemma \ref{lemma:matrixintegrating}) and consider the factorization
$R=FU$ given by Theorem \ref{thm:PSforuse}. Then
\[\B:=F^{-1}XF+F^{-1}\A F\]
is skew-hermitian and $\B\in \Omega_{-1}\oplus\Omega_{0}\oplus\Omega_{1}$. In other words $\B$ determines a pair $(B,\Psi)$ with $B\in \Omega^{1}(M,\mathfrak{u}(n))$ and
$\Psi\in C^{\infty}(M,\mathfrak{u}(n))$.
\label{lemma:miracle}
\end{Lemma}

\begin{proof} Let us differentiate the equation $R=FU$ along the geodesic flow to obtain
\[0 = XR + \A R=(XF) U+FXU + \A FU.\]
Writing $\B:=F^{-1}XF+F^{-1}\A F$, it follows that 
\begin{equation}
\B=-(XU)U^{-1}.
\label{eq:miracle}
\end{equation}
Since $U$ is unitary, we have $U^* = U^{-1}$ and 
\[
((XU)U^{-1})^*= U X(U^{-1}) = -(XU)U^{-1}.
\]
Thus $(XU)U^{-1}$ is skew-hermitian and by \eqref{eq:miracle} so is $\B$. Recall
that $X$ has the mapping property $X:\oplus_{k\geq 0}\Omega_{k}\to \oplus_{k\geq -1}\Omega_{k}$ and
hence since $F$ and $F^{-1}$ are holomorphic $F^{-1}XF\in  \oplus_{k\geq -1}\Omega_{k}$.
Similarly since we are assuming $\A\in \oplus_{k\geq -1}\Omega_{k}$, $F^{-1}\A F\in  \oplus_{k\geq -1}\Omega_{k}$. Thus $\B\in \oplus_{k\geq -1}\Omega_{k}$.
The lemma follows directly from \eqref{eq:miracle} and Lemma \ref{lemma:easy}.
\end{proof}

\begin{Remark}{\rm We can compute the pair $(B,\Psi)$ from the lemma quite explicitly as follows. The defining equation for $\B$ may be re-written as
\[XF+\A F-F\B=0.\]
If we recall that $X=\eta_{-}+\eta_{+}$ we can write the degree $0$ and $-1$ terms as
\[\eta_{-}F_{1}+\A_{-1} F_{1}+\A_{0}F_{0}-F_{1}\B_{-1}-F_{0}\B_{0}=0\]
and
\[\eta_{-}F_{0}+\A_{-1}F_{0}-F_{0}\B_{-1}=0.\]
From these two equations we can solve for $\B_{-1}$ and $\B_{0}$ in terms of $\A_{-1}, \A_{0}, F_{0}$ and $F_{1}$
since $F_{0}$ is easily checked to be invertible. It is interesting to observe that even if we start with $\A=\Phi\in\Omega_0$, so there is no reason for $\B$ to contain only a zero Fourier mode, in fact $\B_{-1}=0$ iff
$\eta_{-}F_{0}=0$ and it is not at all clear how to arrange $R$ for this to happen.

}
\end{Remark}

\begin{Remark}{\rm Since the decomposition $R=FU$ is unique (assuming $U(x,1)=\id$), this means that after fixing $R$ we have a well-defined transformation $\A\mapsto \B$. Once $R$ is fixed, any other smooth integrating factor has the form $RW$ where $W\in C^{\infty}(SM,GL(n,\C))$ is a first integral, i.e.\ $XW=0$.

}
\end{Remark}

We are now ready to prove the following fundamental result for the transport equation. As we already pointed out,
$X$ has the mapping property $X:\oplus_{k\geq 0}\Omega_{k}\to \oplus_{k\geq -1}\Omega_{k}$.
If $\A\in \oplus_{k\geq -1}\Omega_{k}$, the transport operator $X+\A$ retains this property
and the following attenuated version for systems of \cite[Proposition 5.2]{SU11} holds; compare also with \cite[Theorem 6.6]{PSU12}.

\begin{Theorem} Let $(M,g)$ be a simple surface. Let $\A\in C^{\infty}(SM,\mathfrak{gl}(n,\C))$ and assume $\A\in \oplus_{k\geq -1}\Omega_{k}$.
Let $u\in C^{\infty}(SM,\Cm^{n})$ be a smooth function such that $u|_{\partial SM}=0$ and
\[Xu+\A u=-f\in \oplus_{k\geq -1}\Omega_{k}.\]
Then $u$ is holomorphic.
\label{thm:holomorphic}
\end{Theorem}

\begin{proof} From $Xu+\A u=-f$, with $F$ and $\B$ as in Lemma \ref{lemma:miracle}, we deduce after a calculation
\begin{equation}
X(F^{-1}u)+\B(F^{-1}u)=-F^{-1}f
\label{eq:relat}
\end{equation}
and $F^{-1}u|_{\partial SM}=0$. Since $F^{-1}$ is holomorphic, it follows that
$F^{-1}f\oplus_{k\geq -1}\Omega_{k}$. Let
\[q:=\sum_{-\infty}^{-1}(F^{-1}u)_{k}.\]
Then
\[Xq+\B q\in \Omega_{-1}\oplus\Omega_{0}.\]
Since $q|_{\partial SM}=0$ and $\B$ is skew-hermitian, it follows from \cite[Theorem 7.1]{PSU12} (see the beginning of the proof of that theorem) that $q\in\Omega_{0}$, and thus $q=0$.
This implies that $F^{-1}u$ is holomorphic and hence $u=F(F^{-1}u)$ is also holomorphic.
\end{proof}

\begin{Remark} {\rm Note that \eqref{eq:relat} gives
\[I_{\A}(f)=FI_{\B}(F^{-1}f).\]
In principle, this identity together with the methods in \cite{MNP19} could be used to derive stability estimates for the linear problem and via Proposition \ref{eq:pseudo-linearization}, stability estimates for the non-linear problem
as well. Once a stability estimate is established, it is quite likely that the methods in \cite{MNP19} will also deliver a consistent inversion to the statistical inverse problem. We do not pursue this here.

}
\end{Remark}

We can now complete the proof of Theorem \ref{thm:main}.

\begin{proof}[Proof of Theorem \ref{thm:main}]
Consider an arbitrary attenuation pair $(A,\Phi)$, where $A\in \Omega^{1}(M,\mathfrak{gl}(n,\C))$ and
$\Phi\in C^{\infty}(M,\mathfrak{gl}(n,\C))$, and set $\A(x,v)=A(x,v)+\Phi(x)$.
 If $I_{A,\Phi}(f)=0$, by the regularity result \cite[Proposition 5.2]{PSU12} there is a smooth function $u$ such that $u|_{\partial SM}=0$ and
\begin{equation}
Xu+\A u=-f\in \Omega_{-1}\oplus\Omega_{0}\oplus\Omega_{1}.
\label{eq:inproof}
\end{equation}
Since $\A\in \Omega_{-1}\oplus\Omega_{0}\oplus\Omega_{1}$, Theorem \ref{thm:holomorphic} gives that
$u$ is holomorphic. Since the conjugates of both $\A$ and $f$ also belong to $\Omega_{-1}\oplus\Omega_{0}\oplus\Omega_{1}$, conjugating equation \eqref{eq:inproof} and applying Theorem \ref{thm:holomorphic}
again we deduce that $\bar{u}$ is also holomorphic. Thus $u=u_{0}$. If we now set $p:=-u_{0}$ we see that
$p|_{\partial M}=0$ and \eqref{eq:inproof} gives right away $F=\Phi p$ and $\alpha=dp+Ap$ as desired.
\end{proof}

\begin{proof}[Proof of Theorem \ref{thm:nonabelian}] From Proposition \ref{proposition:gaugeU} we know that $C_{A,\Phi}=C_{B,\Psi}$ means that there exists a smooth $U:SM\to GL(n,\C)$ such that $U|_{\partial SM}=\id$ and
\begin{equation} \label{eq:geq}
\B=U^{-1}XU+U^{-1}\A U,
\end{equation}
where $\B(x,v)=B_{x}(v)+\Psi(x)$.  We rephrase this information in terms of an attenuated
ray transform. If we let $W=U-\id$, then $W|_{\partial SM}=0$ and 
\[XW+\A W-W\B=-(\A-\B).\]
Hence $W$ is associated with the attenuated X-ray transform $I_{E(\A,\B)}(\A-\B)$
and if $C_{A,\Phi}=C_{B,\Psi}$, then this transform vanishes.
Note that $\A-\B\in\Omega_{-1}\oplus\Omega_{0}\oplus \Omega_{1}$.

Hence, making the choice to ignore the specific form $E(\A,\B)$, we can apply Theorem \ref{thm:main} to deduce that $W$ only depends on $x$.
Hence $U$ only depends on $x$ and if we set $u(x)=U_{0}$, then \eqref{eq:geq} easily translates into
$B=u^{-1}du+u^{-1}Au$ and $\Psi=u^{-1}\Phi u$ just by looking at the components of degree $0$ and $\pm 1$.
\end{proof}

\section{General Lie groups}\label{section:lieg} Let $(M,g)$ be a compact non-trapping surface with strictly convex boundary. Given an arbitrary Lie group $G$ with Lie algebra $\g$
and $\A\in C^{\infty}(SM,\g)$ we first explain how to make sense of the scattering data (see \cite{Hall,Wa} for  background on Lie groups and Lie algebras).  
If we let $L_{g}$ and $R_{g}$ denote left and right translation by $g$ in the group respectively, we observe
\[d(L_{g^{-1}})|_g:T_{g}G\to T_{e}G=\g.\]
Hence if we set
\[\omega^{L}_{g}(v):=d(L_{g^{-1}})|_g(v),\]
we see that $\omega^L\in \Omega^{1}(G,\g)$. The 1-form $\omega^L$ is called the left Maurer-Cartan 1-form of $G$.
If $G$ is a {\it matrix} Lie group (i.e.\ a closed subgroup of $GL(n,\C)$) then $\omega^L=g^{-1}dg$ where $dg$ is the derivative of the embedding $G \to GL(n,\C)$.
Using $R_{g}$ we can define similarly a right Maurer-Cartan
form $\omega^{R}_g:=d(R_{g^{-1}})|_g$ and for matrix Lie groups this is $(dg) g^{-1}$.

The matrix ODE that determines the non-Abelian X-ray transform may now be written in abstract terms as the unique solution $U:[0, \tau]\to G$ such that
\begin{equation}
 U^*\omega^R(\partial_t) + \A(\varphi_{t}(x,v)) = 0, \qquad U(\tau(x,v)) = e. 
\label{eq:ODEG}
\end{equation}
Thus $C_{\A}:\partial_{+}SM\to G$ is defined as $C_{\A}(x,v)=U(0)$. Note that the ODE may also be written
as $\dot{U}+dR_{U}|_{e}(\A)=0$.

As before, the gauge group $\mathcal G$ is given by those smooth $u:M\to G$
such that $u|_{\partial M}=e$. Given a pair $(A,\Phi)$ with $A\in\Omega^{1}(M,\g)$ and $\Phi\in C^{\infty}(M,\g)$ 
we have an action
\[ (A,\Phi)\cdot u=(u^*\omega^{L}+\text{Ad}_{u^{-1}}(A),\text{Ad}_{u^{-1}}(\Phi)),\]
where $\text{Ad}_g:\g \to \g$ is the Adjoint action (i.e.\ $\text{Ad}_g = d\Psi_g|_e$ where $\Psi_g: G \to G$, $\Psi_g(h) = g h g^{-1}$).
It is straightforward to check that for any $u\in \mathcal G$,
\[C_{(A,\Phi)\cdot u}=C_{A,\Phi}.\]

The main result of this section is:

\begin{Theorem} \label{thm:nonabelian2}
Let $(M,g)$ be a simple surface and let $G$ be an arbitrary Lie group with Lie algebra $\g$.
Suppose we are given pairs $(A,\Phi)$ and $(B,\Psi)$ with
$A,B\in \Omega^{1}(M,\g)$ and
$\Phi,\Psi\in C^{\infty}(M,\g)$. If
\[C_{A,\Phi}=C_{B,\Psi},\]
then there is $u\in \mathcal G$ such that $(A,\Phi)\cdot u=(B,\Psi)$.
\end{Theorem}

\subsection{Matrix Lie groups} Let us first check that using Theorem \ref{thm:nonabelian} we can prove
Theorem \ref{thm:nonabelian2} for an arbitrary matrix Lie group. Namely:

\begin{Proposition} Let $(M,g)$ be a simple surface. Let $G$ be a matrix Lie group.
Suppose we are given pairs $(A,\Phi)$ and $(B,\Psi)$ with
$A,B\in \Omega^{1}(M,\g)$ and
$\Phi,\Psi\in C^{\infty}(M,\g)$. If
\[C_{A,\Phi}=C_{B,\Psi},\]
then there is $u\in \mathcal G$ such that $(A,\Phi)\cdot u=(B,\Psi)$.

\label{prop:nonabelian3}
\end{Proposition}

\begin{proof} Since $G$ is a subgroup of $GL(n,\C)$ we see that $\g\subset \mathfrak{gl}(n,\C)$.
Thus by Theorem \ref{thm:nonabelian} there is $u:M\to GL(n,\C)$ such that $u|_{\partial M}=\id$
and $u\cdot (A,\Phi)=(B,\Psi)$. We only need to check that under these conditions $u$ takes values in fact
in $G$. The gauge equivalence gives (with $e=\id$)
\[du=uB-Au=d(L_{u})|_{e}(B)-d(R_{u})|_{e}(A)\]
and note that since $A$ and $B$ take values in $\g$, for $g\in G$, $d(L_{g})|_e(B)-d(R_{g})|_e(A)\in T_{g}G$.
Fix $x\in M$ and take any curve $\gamma:[0,1]\to M$ connecting $\gamma(0)\in \partial M$
and $\gamma(1)=x$.
Let
\[Y(g,t):=d(L_{g})|_e(B_{\gamma(t)}(\dot{\gamma}(t))-d(R_{g})|_e(A_{\gamma(t)}(\dot{\gamma}(t))\in T_{g}G.\]
This is clearly a time-dependent vector field in $G$. Thus there is a unique solution $g(t)$ to the ODE in $G$,
$\dot{g}=Y(g(t),t)$ with $g(0)=e$. Since $u(\gamma(t))$ solves the same ODE with the same initial condition
we see that $u(x)=g(1)\in G$ as desired.
\end{proof}

\subsection{Lie group coverings} Let us now discuss the behaviour of the scattering data under coverings, as this will prove quite useful for the proof of Theorem \ref{thm:nonabelian2}.

Suppose we have a Lie group covering map
$p:\G\to G$ and $\A,\B\in C^{\infty}(SM,\g)$. Both Lie groups have the same Lie algebra, $p$ is a Lie group homomorphism and $dp|_e:T_{e}\G\to T_{e}G$ realizes the identification between Lie algebras, thus $\A,\B$ can be considered as infinitesimal data for both $G$ and $\G$ (henceforth we will not distinguish between
$\A$ and $(dp|_e)^{-1}(\A)$).

\begin{Lemma} Let $C_{\A}$ denote the scattering data of $G$ and $\widetilde{C}_{\A}$ the scattering data of $\G$. Then $p\,\widetilde{C}_{\A}=C_{\A}$.
\label{lemma:trivial}
\end{Lemma}

\begin{proof} This is an immediate consequence of the fact that
the solutions $U:[0,\tau]\to G$ and $\widetilde{U}:[0,\tau]\to \G$ to the ODEs are related by
$p\widetilde{U}=U$ since for the Maurer-Cartan forms we have $p^*\omega=\widetilde{\omega}$.
\end{proof}

Next we show:

\begin{Lemma} Let a covering $p:\G \to G$ be given. Then $C_{\A}=C_{\B}$ implies
$\widetilde{C}_{\A}=\widetilde{C}_{\B}$.
\label{lemma:lesstrivial}

\end{Lemma}

\begin{proof} Let $U_{x,v}^{\A}:[0,\tau(x,v)]\to G$ denote the unique solution to the ODE \eqref{eq:ODEG} for $\A$ with $U(\tau)=e$. We use similar notation for
$\B$ and $\G$. If $C_{\A}=C_{\B}$, then for all $(x,v)\in \partial_{+}SM$, consider the concatenation of paths in $G$:
\[\Gamma(x,v):=U_{x,v}^{\A}*\text{Inv}(U_{x,v}^{\B}),\]
where $\text{Inv}$ indicates the path traversed in the opposite orientation. The path $\Gamma(x,v)$ is in fact
a closed loop in $G$, thanks to the assumption $C_{\A}=C_{\B}$. These loops depend continuously
on $(x,v)\in \partial_{+}SM$ and if $(x,v)$ is at the glancing (i.e. the region where $v \in T_x(\partial M)$) we get a constant path equal to the identity.
Hence $\Gamma(x,v)$ are all contractible in $G$ and thus the unique lifts $\widetilde{U}_{x,v}^{\A}$, $\widetilde{U}_{x,v}^{\B}$ must have the {\it same} end points. Thus $\widetilde{C}_{\A}=\widetilde{C}_{\B}$
as desired.
\end{proof}

The next lemma exploits the fact that $M$ is a disc.

\begin{Lemma} There exists $u:M\to G$ with $u|_{\partial M}=e$ and $(A,\Phi)\cdot u=(B,\Psi)$, iff
there is $\widetilde{u}:M\to\G$ with $\widetilde{u}|_{\partial M}=e$ and $(A,\Phi)\cdot widetilde{u}=(B,\Psi)$.
\label{lemma:top}
\end{Lemma}

\begin{proof} Since $M$ is simply connected, $u:M\to G$ has a unique lift
$\widetilde{u}:M\to \G$ with $u(x_{0})=e$ for some base point $x_{0}\in\partial M$.
Being a lift means $p\widetilde{u}=u$. Since constant paths lift to constant paths, we must have
$\widetilde{u}|_{\partial M}=e$. If $\widetilde{u}$ exists then $u:=p\widetilde{u}$ fulfills the requirements since
$p$ is a homomorphism.
\end{proof}

\begin{proof}[Proof of Theorem \ref{thm:nonabelian2}] By considering the connected component of $G$ we may assume without loss of generality that $G$ is connected.
By Ado's theorem and the strengthening explained in \cite[Conclusion 5.26]{Hall}, there exists a matrix Lie
group $H$ and a Lie algebra isomorphism $\phi:\g\to \mathfrak{h}$. Let $\G$ be the universal cover of $G$, so that $\G$ is a simply connected Lie group.
By the correspondence theorem between Lie groups and Lie algebras, there exists a unique
homomorphism $F:\G\to H$ such that $dF|_e=\phi$. Moreover, since $\phi$ is an isomorphism, the map $F$
is a covering map (cf. \cite[Chapter 3]{Wa}).

Suppose $C_{A,\Phi}=C_{B,\Psi}$ for $G$. Then by Lemma \ref{lemma:lesstrivial}, the same holds for $\G$
and by Lemma \ref{lemma:trivial} it also holds for the matrix Lie group $H$. By Proposition \ref{prop:nonabelian3}
there exists a smooth $q:M\to H$ such that $q|_{\partial M}=\id$ and $ (A,\Phi)\cdot q=(B,\Psi)$.
By Lemma \ref{lemma:top} the map $q$ gives rise to a smooth $u:M\to G$ such that $u|_{\partial M}=\id$ and $(A,\Phi)\cdot u=(B,\Psi)$ as desired.
\end{proof}

\end{document}